\documentclass{amsart}

\usepackage{amsfonts}
\usepackage{amssymb}
\usepackage{amscd}

\newtheorem{thm}{Theorem}[section]
\newtheorem{lem}[thm]{Lemma}
\newtheorem{prop}[thm]{Proposition}

\newtheorem{cor}[thm]{Corollary}

\newtheorem{ex}[thm]{Example}

\newtheorem{rmk}[thm]{Remark}

\begin{document}

\title[Regularity results for classes of Hilbert C*-modules]{Regularity results for classes of Hilbert C*-modules with respect to special bounded modular functionals}
\author{Michael Frank}
\address{Hochschule f\"ur Technik, Wirtschaft und Kultur (HTWK) Leipzig, Fakult\"at Informatik und Medien, PF 301166, D-04251 Leipzig, Germany.}
\email{michael.frank@htwk-leipzig.de}



\dedicatory{in the memory of Gert Kj{\ae}rg{\aa}rd Pedersen}

\maketitle

\begin{abstract}
Considering the deeper reasons of the appearance of a remarkable counterexample by J.~Kaad and M.~Skeide (2023) we consider situations in which two Hilbert C*-modules $M \subset N$ with $M^\bot = \{ 0 \}$ over a fixed C*-algebra $A$ of coefficients cannot be separated by a non-trivial bounded $A$-linear functional $r_0: N \to A$ vanishing on $M$. In other words, the uniqueness of extensions of the zero functional from $M$ to $N$ is focussed. We show this uniqueness of extension for any such pairs of Hilbert C*-modules over W*-algebras, over monotone complete C*-algebras and over compact C*-algebras. Moreover, uniqueness of extension takes place also for any one-sided maximal modular ideal of any C*-algebra. 
Such a non-zero separating bounded $A$-linear functional $r_0$ exist for a given pair of full Hilbert C*-modules $M \subseteq N$ over a given C*-algebra $A$ iff there exists a bounded $A$-linear non-adjointable operator $T_0: N \to N$ such that the kernel of $T_0$ is not biorthogonally closed w.r.t. $N$ and contains $M$. This is a new perspective on properties of bounded modular operators that might appear in Hilbert C*-module theory. By the way, we find a correct proof of Lemma 2.4 of M. Frank (2002) in the case of monotone complete and compact C*-algebras, but find it not valid for certain particular cases.
\end{abstract}

\section{Introduction}

The theory of Hilbert C*-modules exists for about 70 resp.~50 years since the famous works by I. Kaplansky, resp.~by W.~L.~Paschke and by M.~A.~Rieffel. Nevertheless, a new problem has been discovered which forces to review parts of the theory. Some years ago O.~M.~Shalit and B.~Solel investigated Hilbert subproduct systems of Hilbert product systems, cf.~\cite{Shalit_Solel,Shalit_Skeide}. One core question has been whether there might exist special bounded modular functionals on pairs of Hilbert C*-modules into the C*-algebra of their coefficients: Let $M\subset N$ be two Hilbert $A$-modules over a given C*-algebra $A$ such that the orthogonal complement of $M$ w.r.t. $N$ equals $\{ 0 \}$. Does there exist a non-trivial modular extension $r_0$ of the zero map from $M$ to $A$ to $N$? This question is the analogue of the categorical separation problem for similar pairs of linear spaces of quite different kinds by (sorts of) bounded linear functionals. We refer to J.~Kaad and M.~Skeide \cite{KS} and to V.~M.~Manuilov \cite{M} for initial investigations.

Considering Hilbert spaces and their Hilbert subspaces as a class of examples, there seems to be not any such problem. To indicate one more complicated background of the problem consider maximal one-sided (say, right) norm-closed ideals $D$ of unital C*-algebras $A$. Both $A$ and $D \subset A$ can be considered as (right) Hilbert $A$-modules inducing the necessary algebraic structures from the C*-structures of $A$ in the usual way. Intuitively, for C*-algebras $A$ without finite part, $aD = 0$ for some $a \in A$ should force $a=0$, so the zero functional on $D$ would have only the zero modular functional on $A$ as its continuation, or alternatively in the opposite case, $D$ is even a direct orthogonal summand of $A$ as for matrix algebras or for atomic carrier projections from an occasionally existing finite part of $A$.  However, a proof even for modular maximal right ideals $D$ of $A$ will need a deep dive into C*-theory as will be shown in the last section. Note, that the example of $A=C([0,1])$ and $D=C_0((0,1])$ of all continuous functions on the unit interval and of the subset of functions vanishing at zero would yield the Banach algebra of all bounded continuous functions $D'=C_b((0,1]) \supset A$ on the half-open unit interval $(0,1]$  as the dual Banach $A$-module $D'$ of $D$. So the problem of the extension of the zero functional on $D$ reaches beyond the hosting selfdual Hilbert $A$-module $A$ in this case. Moreover, $D$ coincides with its bidual Banach $A$-module $D''$ in this case, another special situation. However, turning to arbitrary norm-closed ideals of C*-algebras one has to struggle with the defiencies of noncommutative topology.

\begin{ex}
{\rm 
Let $X=[0,1]$ be the unit interval of real numbers equipped with the usual metric topology arising from the absolute value of the difference of two numbers. Consider the set $J$ of all bounded Borel functions $f$ on $X$ such that the set $\{ x \in X :f(x) \not= 0 \}$ is meager. The set $J$ is a two-sided norm-closed ideal in the monotone $\sigma$-complete C*-algebra of all bounded Borel functions on $X$ with its supremum norm. The C*-algebra $A=D(X)$ constructed as the quotient algebra of the latter by $J$ is known as the Dixmier algebra on $X$. It is a monotone complete C*-algebra (hence, AW*) and, as a commutative AW*-algebra, also injective, in fact the injective envelope $I(C(X))$ of the C*-algebra C(X), cf. \cite{Dixmier,SaitoWright_2015,Hamana79,Gonshor}. 

Consider the Hilbert $A$-module $M=l_2(A)$ of all sequences of elements $a_i \in A$ such that the series $\sum_i |a_i|^2$ converges in norm. This module $M$ is countably generated over $A$ by its the orthonormal basis $\{ e_i = (0, ... ,0,1_A^{(i)},0, ...) \}$. The $A$-dual Banach $A$-module $l_2(A)'$ of the Hilbert $A$-module $l_2(A)$ can be identified with the set of sequences of elements $a_i \in A$  such that the sequence of partial sums $\{ \sum_{i=1}^k |a_i|^2 \}_k$ is bounded by a sequence-specific constant from above. The norm of each sequence is derived from the respective least upper bound. 
Note, that these series are well-defined since $A$ is monotone complete, the respective sequence of partial sums $\{ \sum_{i=1}^k |a_i|^2 \}_k$ is norm-bounded, positive and monotone increasing in $A$, what makes the supremum exist as a positive element of $A$. So one can define an $A$-valued inner product on $l_2(A)'$ setting $\langle \{a_i\}_i, \{ a_i \}_i \rangle := \sum_i |a_i|^2$, applying the polarization formula to count the values  $\langle x,y \rangle = \frac{1}{4}  \sum_{k=0}^3 {\bf i}^k \langle x+ {\bf i}^k y, x+ {\bf i}^k y \rangle$ for $x,y \in l_2(A)'$. The Banach $A$-module $N=l_2(A)'$ turns into a self-dual Hilbert $A$-module with an isometrically embedded copy of $M=l_2(A)$. Note, that for any element $f=(a_1, ..., a_i, ...)$ of $l_2(A)'$ one has $f(e_i)=a_i$ for any index $i$.  Also, the orthogonal complement of $l_2(A)$ in $l_2(A)'$ is $\{ 0 \}$. Consequently, every bounded $A$-linear functional $f_0$ on $l_2(A)$ that vanishes on $l_2(A) \subset l_2(A)'$ should be represented by the zero sequence in $l_2(A)'$ and, therefore, has an $A$-valued inner product value $0_A$, i.e. $f_0$ has to be the zero functional.

This shows that the arguments given by V.~M.~Manuilov in Lemma 11 and Corollary 12 of \cite{M} need a thorough revision. The example can be repeated for any compact Hausdorff space $X$ and the injective envelope  $I(C(X))$, cf.~\cite{Gonshor}, as well for any Hilbert $A$-module $M$ and its $A$-dual Banach $A$-module $N=M'$ over monotone complete C*-algebras $A$, cf.~\cite{Hamana92}.
}
\end{ex}

The question of non-trivial modular extensions of the zero functional on the named class of pairs of Hilbert C*-modules is closely related to the problem, whether there are Hilbert C*-modules and bounded C*-linear operators between them, the kernels of which are not biorthogonally complemented, or not. The latter question was investigated by J.~Kaad and M.~Skeide in \cite{KS} in 2021 giving a class of examples with this new phenomenon. Also, V.~M.~Manuilov considered this circle of problems in \cite{M}. Therefore, we cannot hope simply to extend our understanding of the Hilbert space situation and of the maximal norm-closed ideal type examples to, for example, all Hilbert C*-modules over monotone complete C*-algebras. We need a separate investigation which is done for Hilbert C*-modules over W*-algebras, over monotone complete C*-algebras and over compact C*-algebras in the present paper. By the way, we find a correct proof of \cite[Lemma 2.4]{F_2002} in the monotone complete and compact C*-case, but disproved it for certain C*-algebras, cf.~\cite{KS}. Generally speaking, for these classes of Hilbert C*-modules the discussed situation is pretty much similar to that one of the class of Hilbert spaces. In the present paper we do not cite any fact from \cite{F_2002} to avoid any influence from the unproven \cite[Lemma 2.4]{F_2002} on the present explanations.

\section{Some basic definitions and facts}

Our basic references for facts on Hilbert C*-modules are \cite{BG_2002,Paschke,Frank_1990,F_1995,Wegge-Olsen,BMSh_1994,Lance_95,F_1995,Li_10} and others.
To start with, we give some basic definitions and arguments to introduce to the circle of problems treated. The basic structures are Hilbert C*-modules, i.e.~(non-unital, in general) C*-algebras $A$ and (right, Banach) $A$-modules $M$ such that there exist an $A$-valued inner product $\langle .,. \rangle: M \times M \to A$ compatible with both the complex structures on $A$ and on $M$ such that 
\begin{enumerate}
\item $\langle z, x a+y \rangle = \langle z,x \rangle a + \langle z,y \rangle$ for any $a \in A$, $x,y,z \in M$,
\item $\langle x,y \rangle = \langle y,x \rangle^*$ for any $x,y \in M$,
\item $\langle x,x \rangle \geq 0$ for any $x \in M$,
\item $\langle x,x \rangle =0$ iff $x=0$ in $M$.
\end{enumerate}
That is, $\langle .,. \rangle$ is a conjugate $A$-bilinear mapping. We treat only modules that are complete with respect to the derived norm $\|x\| := \|\langle x,x \rangle \|^{1/2}_A$ for $x \in M$. But, this is not precisely enough. We have to consider Hilbert C*-modules always as pairs of a module and of its C*-valued inner product, cf.~\cite{Frank_1999} and Lemma \ref{lemma_basic2} below. 

Standard examples are Hilbert spaces over $\mathbb C$ or norm-closed right ideals of C*-algebras $A$. Often the free projective modules $A^n$ of all $n$-tuples of elements of $A$ or the standard countably generated $A$-module $l_2(A)$ of all sequences of elements of $A$ for which the respective inner product series converge in norm are considered, cf.~the Swan-Serre and Kasparov theorems. Generally speaking, full Hilbert C*-modules $M$ over some C*-algebra $A$ are at the same time (left) Banach modules over their C*-algebra ${\rm K}_A(M)$ of ''compact'' module operators over them. Moreover, this relation is symmetric since there exists a ${\rm K}_A(M)$-valued inner product on $M$ inducing the same norm, and ${\rm K}_{{\rm K}_A(M)}(M)$ equals to $A$. This leads to (strong) Morita equivalence of C*-algebras. 

However, Hilbert C*-modules in particular examples can be much more different in usually supposed ''good'' properties. Common knowledge is the sometimes missing self-duality of Hilbert C*-modules and the possible non-adjointability of some bounded module operators. Less known is the possible existence of two (or more) C*-valued inner products on certain Hilbert C*-modules inducing equivalent norms on them, but which are not unitarily equivalent or similar via a bounded adjointable (or even modular) invertible operator, what may turn the properties of a bounded module operator to be ''compact'' or to be adjointable into a relative property. Thinking further, modularly generating sets of elements are also affected in their possible property to be a standard modular frame by this effect (cf.~\cite[Cor.~6.6]{FL}), but there exist quite regularly Hilbert C*-modules that do not admit even this weak form of a ''nice'' generating set, not talking about kinds of orthonormal bases, cf.~\cite{Li_10,Asadi_et_al}.  So, the point is to identify classes of ''good'' Hilbert C*-modules. Candidates are the class of Hilbert C*-modules over compact C*-algebras, i.e.~C*-algebras that admit a faithful $*$-representation in some C*-algebra of (all) compact operators on a Hilbert space. Another good choice are the classes of Hilbert C*-modules over von Neumann (i.e.~W*-)algebras or over monotone complete C*-algebras. We are going to obtain further facts for these classes adding more evidence. 

The following fact is non-obvious, but very useful for insights: Let $N$ be a Hilbert C*-module over a C*-algebra $A$. Denote by $\langle N,N \rangle$ the norm-closed $A$-linear hull of all $A$-valued inner product values of elements of $N$ in $A$. Moreover, the sets $N$ and $AN$ coincide and any element $n \in N$ can be represented as $n=a x$ for certain elements $a \in A$, $x \in N$ by the Cohen-Hewitt factorization theorem. This fact has been mentioned by several authors in the context of Banach C*-modules and in more general contexts, cf. e.g.~\cite[Thm.~4.1]{Pedersen_1998} or \cite[Thm.~32.22]{Hewitt_Ross_1970}, also \cite[Thm.~II.5.3.7]{Blackadar}. This is one reason why Hilbert $A$-modules $N$ can be considered as Hilbert ${\mathbb M}(A)$-modules over the multiplier algebras ${\mathbb M}(A)$ of $A$, similarly as Banach $A$-modules over the (right) multiplier algebra ${\mathbb R}{\mathbb M}(A)$ of $A$. Recalling Morita equivalence and the symmetry of both the module actions, $N$ is also a Hilbert ${\rm End}^*_A(N)$-module over the C*-algebra of all bounded adjointable operators ${\rm End}^*_A(N) = {\mathbb M}({\rm K}_A(N))$, as well as a Banach ${\rm End}_A(N)$-module over the Banach algebra ${\rm End}_A(M) = {\mathbb L}{\mathbb M}({\rm K}_A(N))$, cf.~\cite{Green_1978,Kasp_1980,Lin_1992}. 

Note, that we consider $A$-dual Banach $A$-modules $N'$ of Hilbert $A$-modules $N$ as right $A$-modules, too, defining $ra$ for (right $A$-linear elements) $r \in N'$ and $a \in A$ by $(ra)(x):= a^* \cdot r(x)$ for any $x \in N$, cf.~\cite[p.~450]{Paschke}. Whenever $A$-valued inner products can be extended from $N$ to $N'$ this simplifies the isometric embedding of $N$ into $N'$ treating elements of both these modules in the same way in formulae. 


\section{Extension of $C^*$-linear functionals: the monotone complete C*-case}

In this section we investigate the question whether the zero bounded C*-linear maps of Hilbert C*-submodules over monotone complete C*-algebras $A$ to their C*-algebra of coefficients could be continued by a non-zero bounded C*-linear functional on the hosting Hilbert C*-module over the same monotone complete C*-algebra in case the orthogonal complement of the submodule is trivial. We treat the cases of W*-algebras and of monotone complete C*-algebras as C*-algebras simultaneously despite of the different spheres of application and partially different techniques. This approach should help readers without deeper knowledge on non-W*, monotone complete C*-algebras to understand the arguments.

We need some intrinsic characterization of selfduality of Hilbert C*-modules over monotone complete C*-algebras. In \cite{F_1995} some kind of order type convergence in such Hilbert C*-modules has been introduced based on order convergence in monotone complete C*-algebras (cf.~\cite{KP},  \cite[Section 1]{Hamana82}). Note, that order convergence in monotone complete C*-algebras (denoted by ${\rm LIM}$ in the sequel) is not supported by any locally convex Hausdorff topology that preserves all algebraic C*-algebra structures, generally speaking (cf.  \cite{Floyd}, \cite[Remark on p. 67]{F_1995}). But, for W*-algebras the w*-topology supports order convergence. Let $A$ be a monotone complete C*-algebra, $M$ be a Hilbert $A$-module and $I$ be a net for indexing. A norm-bounded set $\{ x_\alpha :\alpha \in I \}$ of elements of $M$ is fundamental in the sense of $\tau_2 ^o$-convergence iff the limits ${\rm LIM} \{y , x_\alpha - x_\beta \rangle: \alpha \in I\}$ exist for every $\beta \in I$, any $y \in M$, and the limits ${\rm LIM} \{ {\rm LIM} \{\langle y , x_\alpha - x_\beta \rangle: \alpha \in I\} : \beta \in I\}$ exist for any $y \in M$, too, and equal to zero. Such a set has the $\tau_2^o$-limit $x \in M$ iff the limits ${\rm LIM} \{y , x_\alpha - x \rangle: \alpha \in I\}$ exist for any $y \in M$ and equal to zero. This $\tau_2^o$-convergence respects the module structures and preserves norm-bounded balls, cf. \cite[Def.~2.3, Lemma 2.4]{F_1995}. Considering the W*-case a Hilbert W*-module $M$ is self-dual iff its unit ball is complete with respect to the topology generated by the set $\{ f(\langle x,. \rangle) : f \in A_{*}, x \in M \}$. This set generates the w*-topology on the $A$-dual Hilbert $A$-module $M'$ of $M$, cf.~\cite[Prop.~3.8, Remark 3.9]{Paschke}, \cite[Def.~3.1, Thm.~3.2]{Frank_1990}.


\begin{lem} \label{lemma_basic2}
Let $A$ be a C*-algebra and $M \subseteq N$ be two full Hilbert $A$-modules. Suppose that $M \subseteq N$ has the orthogonal complement $M^\bot_N=\{0\}$ with respect to $N$. Then:
\begin{enumerate}
\item Two different elements $n_1, n_2 \in N$ restricted to $M \subseteq N$ realize pairwise distinct bounded $A$-linear functionals $<n_1, . >$, $<n_2, .>$ on $M$.  
\item In case of the existence of a non-zero bounded $A$-linear functional $r_0: N \to A$ such that $r_0$ vanishes on $M$, the entire module $\{ r_0a : a \in A \}$ as well as its norm-closure in $N'$ represents the zero functional on $M$. So $r_0$ cannot be represented by an element of $N$ via the $A$-valued inner product on $N$ by supposition.
\item Let $r_0$ as in (ii). The Hilbert $A$-module $N$ admits (at least) two $A$-valued inner products inducing equivalent norms, whose reductions to $M$ coincide. The more, these two $A$-valued inner products are not equivalent, i.e. $\langle T(.),. \rangle^{(1)} \not \equiv \langle .,. \rangle^{(2)}$ for any positive w.r.t. $\langle .,.\rangle^{(1)}$, bounded bijective $A$-linear operator $T$ on $N$. 
\item Given the situation at (iii), the notions of bounded module operators on $N$ to be ''compact'' or to be adjointable w.r.t. one of these two named $A$-valued inner products depend on the choice of the $A$-valued inner product on $N$ with identical reductions to $M$.
\end{enumerate}
\end{lem}

\begin{proof}
To see the assertion (i) select $n_1, n_2 \in N$ and consider the difference $n_1-n_2$. Suppose, both $n_1$ and $n_2$ induce the same bounded $A$-linear map from $M$ to $A$ taking $<n_i,.>$ and reducing it to $M\subseteq N$. Then $n_1-n_2=0$ by supposition, i.e. $n_1=n_2$ in $N$ follows. 

In case there exists an element $a \in A$ such that $r_0a \in N$ and $r_0a \not= 0$ we would have a non-zero element of $N$ giving a zero $A$-valued functional on $M$ via $\langle r_0a,. \rangle$, a contradiction to the supposition.

Now, set $\langle .,. \rangle^{(2)} := \langle .,. \rangle^{(1)} + r_0^*(.)r_0(.)$. Obviously, it is an $A$-valued inner product on $N$ whose reduction to $M$ gives the initial $A$-valued inner product back. By general operator inequalities (\cite[Prop.~2.6]{Paschke}) we have
\[
\langle n,n \rangle^{(1)} \leq \langle n,n \rangle^{(2)} \leq (1+\|r_0\|^2) \langle n,n \rangle^{(1)} \, .
\]
This shows the equivalence of the induced norms on $N$. Now, suppose $\langle T(.),. \rangle^{(1)} \equiv \langle .,. \rangle^{(2)}$ for some positive w.r.t. $\langle .,. \rangle^{(1)}$,  bounded bijective $A$-linear operator $T$ on $N$. (This is equivalent to the existence of some adjointable  w.r.t. $\langle .,. \rangle^{(1)}$, bounded bijective operator $S$ on $N$ with $\langle S(.),S(.) \rangle^{(1)} \equiv \langle .,. \rangle^{(2)}$.) We obtain the functional equality $\langle (T-{\rm id}_N)(n), . \rangle^{(1)} = r_0^*(n)r_0(.)$ for any $n \in N$. Therefore, the bounded $A$-linear functional $r_0^*(n)r_0(.) \in N'$ is represented by a non-zero element of $N$ for any fixed $n \in N$. This contradicts assertion (ii) since there are non-zero functionals in this set $\{ r_0^*(n)r_0(.) ; n \in N \}$ by supposition. 

The last assertion follows from \cite[Thm.~3.7, Prop.~5.3]{Frank_1999}.
\end{proof}


\begin{lem} \label{lemma_basic}
Let $A$ be a monotone complete C*-algebra and $M \subseteq N$ be two Hilbert $A$-modules. Suppose that $M \subseteq N$ has the orthogonal complement $M^\bot_N=\{0\}$ with respect to $N$. Then:
\begin{enumerate}
\item The C*-algebras $\langle M,M \rangle$ and $\langle N,N \rangle$ in $A$ have the same central carrier projection $p \in A$, i.e. both their annihilators with respect to $A$ equal to \linebreak[4] 
$(1-p)A$ with $p \in Z(A)$. Obviously, both $\langle M,M \rangle$ and $\langle N,N \rangle$ are two-sided norm-closed ideals of the monotone complete C*-algebra $pA$, as well as $\langle M,M \rangle$ is a two-sided ideal of $\langle N,N \rangle \subseteq pA$. 
\item The multiplier C*-algebras of their centers equal to $pZ(A)$, so the unitizations of $\langle M,M \rangle$ and of $\langle N,N \rangle$ inside $pA$ share the identity element $p$ of $pZ(A) \subseteq pA$ as their respective identities.  
\item The multiplier C*-algebras of both $\langle M,M \rangle$ and $\langle N,N \rangle$ equal to $pA$, so the centers of their multiplier C*-algebras both equal to $pZ(A)$, too. (In general, $Z({\mathbb M}(A))$ can be larger than ${\mathbb M}(Z(A))$.) So the two C*-algebras $\langle M,M \rangle$ and $\langle N,N \rangle$ share the identity element $p \in pZ(A)$ of their unitizations even in this sense. 
\item The centers of the C*-algebras $End_A^*(M)$ and $End_A^*(N)$ of all (bounded) adjointable $A$-linear operators on $M$ and on $N$, respectively, are isometrically $*$-isomorphic to $pZ(A) \subseteq pA$.
\end{enumerate}
\end{lem}

\begin{proof}
By \cite[Prop.~8.2.2]{SaitoWright_2015} every subset $S$ of an AW*-algebra $A$ admits a right and a left annihilator set of the form $p_rA $ and $Ap_l $ where $p_l, p_r \in A$ are orthogonal projections. If $S$ is a two-sided (norm-closed) ideal of $A$, then $p_l=p_r $ is a central projection in $A$. Whenever the central supports $(1-p_M)$ and $(1-p_N)$ of $\langle M,M \rangle$ and of $\langle N,N \rangle$, respectively, would be different, the subset $(p_M-p_N)A$ would act non-trivially on $N$ and trivially on $M$. So, the subset $(p_M-p_N)N$ would be orthogonal to $M \subseteq N$, but it has to consist of the zero element only by supposition. This is only possible, if $p_M=p_N$. The set $\langle M,M \rangle$ is obviously a subset and a C*-subalgebra of the C*-algebra $\langle N,N \rangle \subseteq pA$, so it is a two-sided ideal in $\langle N,N \rangle$ because it is a two-sided ideal of $pA$.

To derive the next two statements we make use of \cite[Theorem]{Pedersen_1984}: Let $B$ be a C*-subalgebra of an AW*-algebra $A$ with zero annihilator of $B$ in $A$. Then the set of two-sided multipliers of $B$ in $A$ is isometrically $*$-isomorphic to the set of double centralizers of $B$ (i.e., the set ${\mathbb M}(B)$ of multipliers of $B$ in its enveloping von Neumann algebra $B'' \equiv B^{**}$) via an isomorphism that extends the identity map on $B$. Assertion (ii) follows if we consider the commutative (monotone complete) AW*-algebra $pZ(A)$, where $Z(\langle M,M \rangle) \subseteq Z(\langle N,N \rangle) \subseteq pZ(A)$. Commutativity and zero annihilators give ${\mathbb M}(Z(\langle M,M \rangle)) = {\mathbb M}(Z(\langle N,N \rangle)) = pZ(A)$ by the cited theorem. To obtain assertion (iii) we replace the centers by the entire respective C*-algebras. So, ${\mathbb M}(\langle M,M \rangle) = {\mathbb M}(\langle N,N \rangle) = pA$ by the ideal properties, and therefore, $Z({\mathbb M}(\langle M,M \rangle)) = Z({\mathbb M}(\langle N,N \rangle)) = pZ(A)$. 

By results by P.~Green \cite[Lemma 16]{Green_1978} and by G.~G.~Kasparov \cite[Th. 1]{Kasp_1980} we can identify the set of all bounded adjointable (module) maps on a given Hilbert C*-module isometrically $*$-isomorphically with the multiplier algebra of the C*-algebra of all ''compact'' module operators on it. Obviously, the set $\{ x \cdot {\rm id}_M : x \in Z(M(A)) \}$ is contained in the center of $End_A^*(M)$. Considering $M$ as an $\langle M,M \rangle$-$K_A(M)$ equivalence bimodule (cf.~\cite[Section 1]{BMSh_1994}) we can interchange the roles of $\langle M,M \rangle$ and of $K_A(M)$ as primary and secondary C*-algebras of coefficients of a respectively (double-)full Hilbert C*-bimodule.  Hence, the center of $End_A^*(M)$ is canonically contained in the center of $pA$ by (ii) and (iii) above. So, they have to coincide with the center of $pA$, $pZ(A)$. The argument for $N$ is similar. The centers of $End_A^*(M)$ and of $End_A^*(N)$ turn out to be isometrically *-isomorphic to $pZ(A)$, and to each other.
\end{proof}


The following theorem has parallels in the result \cite[Thm.~13]{M} by V.~M.~Manuilov which was stated for the communtative W*-case:

\begin{thm} \label{thm_isometric}
Let $A$ be a monotone complete C*-algebra and $M \subseteq N$ be two Hilbert $A$-modules. Suppose that $M \subseteq N$ has the orthogonal complement $M^\bot_N=\{0\}$ with respect to $N$.
Then the selfdual Hilbert $A$-module $M'$ admits an isometric embedding as a Hilbert $A$-submodule of the selfdual Hilbert $A$-module $N'$, which extends the given isometric embedding of $M$ in $N$ in the same way as $M$ is isometrically embedded in $M'$. 
The embedded copy of $M'$ in $N'$ is an orthogonal direct summand. 
\end{thm}

\begin{proof}
We need a type of order convergent nets in $M$ to give an intrinsic characterization of selfduality for Hilbert C*-modules over monotone C*-algebras. By \cite[Lemma 3.7]{Lin_1992}, for Hilbert C*-modules $M$ over monotone complete C*-algebras $A$ the $A$-valued inner product on $M$ can be continued to an $A$-valued inner product on the dual Banach $A$-module $M'$ such that $M$ is isometrically embedded in $M'$ by the map $x \in M \to \langle x,. \rangle \in M'$ preserving the inner product values, $r(x)=\langle r,x \rangle$ for any $x \in M \subseteq M'$, $r \in M'$, and $\langle r,r \rangle =\sup\{ r(x)^*r(x) : x \in M, \|x\| \leq 1 \}$. By \cite[Thm.~4.1]{F_1995} a Hilbert $A$-module over a monotone complete C*-algebra $A$ is self-dual iff the unit ball of $M$ is complete with respect to $\tau_2^o$-convergence.

Denote by $M^\dagger$ the $\tau_2^o$ -completed canonical copy of $M$ in $M'$, which is also isometrically embedded in $M'$ by construction. It has to be self-dual, and so either $M^\dagger = M'$ or $(M^\dagger)^\bot \not= \{ 0 \}$ in $M'$ since self-dual Hilbert C*-submodules are always direct summands. The latter would force $M^\bot \not= \{ 0 \}$ in $M'$, a contradiction to the definition of $M'$.  In particular, $M'$ does not contain any non-trivial elements perpendicular to the submodule $M$. 

Consider the isometric embedding of $M$ into $N$, which can be seen as an isometric embedding of $M$ into the self-dual Hilbert $A$-module $N'$ via the canonical isometric embedding of $N$ into its $A$-dual $N'$. By definition isometric module embeddings preserve the module structure and the norm of each element. By \cite{Lance} and \cite{Blecher} surjective module isometries of Hilbert C*-modules preserve the C*-inner product values, cf.~\cite[Thm.~5]{F_1997} and \cite[Thm.~1.1]{Solel}. Since $N'$ is a self-dual Hilbert $A$-module the biorthogonal complement $M^{\bot\bot}$ of $M$ embedded in $N'$ is also a self-dual Hilbert $A$-submodule and direct orthogonal summand of $N'$. The more, $M^\bot_{M^{\bot\bot}} = \{ 0 \}$ by construction. 

So for the isometric embedding of $M$ into $N'$ we can repeat the process of $\tau_2^o$-completion canonically restricting to elements of its biorthogonal completion $M^{\bot\bot}$ with respect to $N'$, without knowing the nature of its orthogonal completion $M^{\bot}$ with respect to $N'$. For the $\tau_2^o$-completion of $M$ in $M^{\bot\bot} \subseteq N'$ we obtain an isometric embedding of $M^\dagger$ into $N'$. It has to be self-dual by \cite[Thm.~4.1]{F_1995} and it has to coincide with $M_{N'}^{\bot\bot}$, and so either $M^\dagger = N'$ or $(M^\dagger)^\bot \not= \{ 0 \}$ in $N'$ since self-dual Hilbert C*-submodules are always direct summands. Also,  $M^\dagger$ is isometrically isomorphic to $M'$ as a Hilbert $A$-module. We denote the orthogonal, positive projection of $N'$ onto the isometrically embedded copy of $M'$ by $P$.

Consider the elements $y \in N$ as $A$-linear functionals $\langle y,. \rangle_N$ on the Hilbert $A$-submodule $M$. Two such elements $y_1,y_2$ induce the same bounded $A$-linear map $\langle y_1,. \rangle \equiv \langle y_2,. \rangle$ on $M$ if  and only if their difference is the zero element of $N$ since $M^\bot_N = \{ 0 \} $ by supposition. So, the elements of $N$ can be identified with elements of $M' = M^\dagger \subseteq N'$ injectively. Moreover, any element of $N \setminus M$ acts in another way on $M$ as any element of $M$. 
\end{proof}

In the following, we shall show that $P={\rm id}_{N'}$.


\begin{lem} \label{lemma_P}
Let $A$ be a monotone complete C*-algebra and $M \subseteq N$ be two Hilbert $A$-modules. Suppose that $M \subseteq N$ has the orthogonal complement $M^\bot_N=\{0\}$ with respect to $N$.
Then any orthogonal, positive projection $P:N' \to N'$ with $P \not= {\rm id}_N$ and $M' \subseteq P(N')$ is not an element of the multiplier C*-algebra $End_A^*(N)$ of the C*-algebra of ''compact'' module operators $K_A(N)$ of the isometrically embedded in $N'$ copy of $N$, i.e. any such $P \not= {\rm id}_N$ with $M' \subseteq P(N')$ is not a bounded adjointable (module) operator on $N$.
\end{lem}

\begin{proof}
Again, by results by P.~Green \cite[Lemma 16]{Green_1978} and by G.~G.~Kasparov \cite[Th. 1]{Kasp_1980} we can identify the set of all bounded adjointable (module) maps on a given Hilbert C*-module isometrically with the multiplier algebra of the C*-algebra of all ''compact'' module operators on it. So, if $P$ would be a non-one two-sided multiplier of the C*-algebra $K_A(N)$  then $P$ would belong to ${\mathbb M}(K_A(N)) = End_A^*(N)$. That is, $P(N) =M$ would have a non-trivial orthogonal complement $({\rm id}_N-P)(N)$ in $N$, a contradiction to the supposition.
\end{proof}


\begin{prop} \label{prop_compact}
Let $A$ be a monotone complete C*-algebra and $M \subseteq N$ be two Hilbert $A$-modules. Suppose that $M \subseteq N$ has the orthogonal complement $M^\bot_N=\{0\}$ with respect to $N$.
The four operator norms $\|T\|_M$, $\|T\|_N$, $\|T\|_{M'}$ and $\|T\|_{N'}$ coincide for any ''compact'' operator $T \in K_A(M)$ realized on resp. $M$, $N$, $M'$ and $N'$. As a consequence, the given isometric modular embedding $M \subseteq N$ gives rise to an isometric $*$-representation of $K_A(M)$ in $K_A(N)$ on the level of ''compact'' modular operators on $M$ and on $N$, respectively.
\end{prop}

\begin{proof}
One knows $M$ is a Hilbert $A$-submodule of $N$, of $M'$ and of $N'$, resp., via canonical isometric modular embeddings. Only $M'$ and $N'$ are a priori self-dual. So, any elementary ''compact'' operator $\theta_{x,y}(\cdot) = y \langle x,. \rangle$ with $x,y \in M$ can be uniquely continued to a ''compact'' operator on its self-dual hosting Hilbert $A$-module with the same operator norm on the self-dual hosting Hilbert $A$-module, cf. \cite[Prop.~3.6]{Paschke} for the W*-case and \cite[Cor.~6.3]{F_1995} for the monotone complete C*-case. Of course, this continuation can be considered as $\theta_{x,y}$ again, with $x,y$ of the isometrically embedded copy of $M$, otherwise the operator norm would increase. So we have the operator norm equalities $\|\theta_{x,y}\|_M = \|\theta_{x,y}\|_{M'}$, $\|\theta_{x,y}\|_M= \|\theta_{x,y}\|_{N'}$. For analogous reasons, $\|\theta_{x,y}\|_N= \|\theta_{x,y}\|_{N'}$ for the particular case of the canonical isometric modular embedding of $N$ into $N'$. Consequently, $\|\theta_{x,y}\|_M = \|\theta_{x,y}\|_N$ for any $x,y \in M$ and for the given isometric modular embedding $M \subset N$ fulfilling the supposition.

Since the elementary ''compact'' operators in $K_A(M)$ can be isometrically identified with elementary ''compact'' operators in $K_A(N)$ which admit the two generating elements from $M \subset N$ we can continue this isometric identification to finite sums $T:=\sum_i \lambda_i \theta_{x_i,y_i}$ with elements $\{ x_i, y_i \}\in M \subset N$ and complex numbers $\{\lambda_i\}$. Indeed, for the chain of isometric modular embeddings $M \subset N \subset N'=M'$ we have
\begin{eqnarray*} 
   \sup_{z \in N, \|z\| \leq 1} \|T(z)\|_N    & = &   \sup_{z \in N', \|z\| \leq 1} \|T(z)\|_{N'} \\
                                                              & = &   \sup_{z \in M', \|z\| \leq 1} \|T(z)\|_{M'} \\
                                                              & = &   \sup_{z \in M, \|z\| \leq 1} \|T(z)\|_M \\
\end{eqnarray*}
Further, we obtain any ''compact'' element $T \in K_A(M)$ and of $K_A(N)$ as the norm-limit of sequences of such finite sums, which can be step by step isometrically identified with resp. ''compact'' operators on $M$ and on $N$. Because of the analogous algebraic structures in $K_A(M)$ and in $K_A(N)$ and because of the demonstrated isometric identifications, we get an isometric $*$-representation of $K_A(M)$ in $K_A(N)$. 
\end{proof}

\begin{prop}  \label{prop_anihilators}
Let $A$ be a monotone complete C*-algebra and $M \subseteq N$ be two Hilbert $A$-modules. Suppose that $M \subseteq N$ has the orthogonal complement $M^\bot_N=\{0\}$ with respect to $N$.
Then $End_A^*(N)$ does not contain any non-zero element $T$ such that $T$ is perpendicular to $K_A(M) \subseteq K_A(N)$. 
\end{prop}

\begin{proof}
Suppose, $\theta_{x,y} T = 0$ on $N$ for any $x,y \in M$, i.e. $y\langle x, T(z) \rangle=0$ for any $z \in N$, any $x,y \in M$. Since $y \in M$ is arbitrary and Lemma \ref{lemma_basic},(i) holds this is equivalent to the condition $\langle x, T(z) \rangle =0$  for any $x \in M$, any $z \in N$. By supposition this means $T(z)=0$ for any $z \in N$. So $T=0$.  If one investigates the opposite multiplication order $T \theta_{x,y}=0$ we can apply the involution on $End_A^*(N)$ and reduce the problem to the one treated.
\end{proof}


\begin{prop} \label{prop_identity}
Let $A$ be a monotone complete C*-algebra and $M \subseteq N$ be two Hilbert $A$-modules. Suppose that $M \subseteq N$ has the orthogonal complement $M^\bot_N=\{0\}$ with respect to $N$.
Then the identity operator on $M$ extends to the identity operator on $N$. This is its unique extension as a bounded modular adjointable operator which preserves the norm.
\end{prop}

\begin{proof}
Consider the isometric modular embeddings $M \subset N$ as given, $N \subseteq N'$ and $M \subseteq M'$ by \cite[Thm.~3.2]{Paschke} and $M' \subseteq N'$ as described in Theorem \ref{thm_isometric}. Since every bounded module operator on the smaller Hilbert $A$-module of each of the last three pairings admits a unique extension to a bounded module operator on the larger Hilbert $A$-module preserving its norm by \cite[Cor.~3.7]{Paschke}, we can extend the identity operator on $M$ to a unique bounded module operator on $N$ which is an orthogonal positive projection $Q$ on $N$. Obviously, $\|Q\| = \|{\rm id}_N\| = 1$, so by the uniqueness of the extensions $Q=P \in End_A^*(N')$ with $P$ from the end of the proof of Theorem \ref{thm_isometric}. However, $P \not\in End_A^*(N)$ by Lemma \ref{lemma_P} in case it is not the identity of $End_A^*(N)$. So, $Q = P = {\rm id}_{N}$ is the only possible alternative according to Theorem \ref{thm_isometric}, Proposition \ref{prop_anihilators}. In other words,  $End_A^*(M)$ and  $End_A^*(N)$ share their identity operators in this canonical setting.
Moreover, by \cite[Thm.~3.2]{Paschke} and \cite[Cor.~3.7]{Paschke} we have ${\rm id}_M = {\rm id}_{M'}$ and ${\rm id}_N = {\rm id}_{N'}$ for the respective identity operators for the canonical isometric modular embeddings. Therefore, ${\rm id}_{M'} = {\rm id}_{N'}$ and $P={\rm id}_{N}$.

By \cite[Thm.]{Pedersen_1984} and Lemma \ref{lemma_P} any non-degenerated isometric $*$-representation of $K_A(N)$ is at the same time a non-degenerated isometric $*$-representation of $K_A(M)$, and hence, an isometric $*$-representation of their respective multiplier algebras $End_A^*(M)$ and  $End_A^*(N)$ which share their identity elements. So, this picture is highly stable. The derived facts of representation theory support the conclusion.
\end{proof}


We arrive at a central result, a particular case of which was published by V.~M.~Ma\-nuilov as \cite[Thm.~9]{M} for commutative W*-algebras and as \cite[Thm.~10]{M} for type I von Neumann algebras..

\begin{thm} \label{theorem_ok}
Let $A$ be a monotone complete C*-algebra and $M \subseteq N$ be two Hilbert $A$-modules. Suppose that $M \subseteq N$ has the orthogonal complement $M^\bot_N=\{0\}$ with respect to $N$.
Then the selfdual Hilbert $A$-module $M'$ admits an isometric embedding as a Hilbert $A$-submodule of the selfdual Hilbert $A$-module $N'$ such that $M'$ coincides with $N'$ preserving the given isometric inclusions. In particular, there does not exist any bounded $A$-linear functional $r_0: N \to A$ such that $r_0$ vanishes on $M$, but which is not the zero functional on $N$.
\end{thm}

\begin{proof}
By Theorem~\ref{thm_isometric} the isometric embedding of $M$ into $N$ can be continued to an isometric embedding of $M'$ into $N'$ as a direct orthogonal summand. Since $M^\bot = \{0 \}$ Lemma \ref{lemma_P} and Proposition \ref{prop_identity} imply the isometric modular isomorphism $M'=N'$. By the definition of $M'$ the (bounded $A$-linear) zero functional on $M$ with values in $A$ has only the zero functional from $N$ to $A$ as its continuation.
\end{proof}

The following fact is not true for any C*-algebra $A$ and any Hilbert $A$-module $N$, what makes it remarkable.

\begin{cor} \label{cor_ok1}
Let $A$ be a monotone complete C*-algebra and $M \subseteq N$ be two Hilbert $A$-modules. Suppose that $M \subseteq N$ has the orthogonal complement $M^\bot_N=\{0\}$ with respect to $N$.
Then the $A$-dual Banach $A$-modules of $N$ and of $M$ isometrically coincide as Banach $A$-modules. In particular, for a given Hilbert $A$-module $N$ this holds for any (smaller-equal) Hilbert $A$-submodule $M$ with $M^\bot_N=\{0\}$. 
\end{cor}

As a central result we got that any isometric embedding of a Hilbert $A$-module $M$ into another Hilbert $A$-module $N$ with $M^\bot = \{ 0 \}$ continues to an isometric coincidence of the $A$-dual Banach $A$-modules $M'$ and $N'$. So the following corollary excludes the existence of examples in the described context like those given by J.~Kaad and M.~Skeide in \cite{KS}. 

\begin{cor}
Let $A$ be a monotone complete C*-algebra and $M \subseteq N$ be two Hilbert $A$-modules. Suppose that $M \subseteq N$ has the orthogonal complement $M^\bot_N=\{0\}$ with respect to $N$.
The for any element $n \in N \setminus M$ the equality $\|n\|_N = \|n\|_{M'}$ holds for these two norms of it. 
\end{cor}

\begin{rmk} {\rm
Let $A$ be a monotone complete C*-algebra and $M \subseteq N$ be two Hilbert $A$-modules. Suppose that $M \subseteq N$ has the orthogonal complement $M^\bot_N=\{0\}$ with respect to $N$.  
Then by the unique extension theorem by W.~L.~Paschke \cite[Cor.~3.7]{Paschke} any bounded (adjointable) operator on $M$ or on $N$, respectively, has a unique equal-norm extension to a bounded adjointable operator on the $A$-dual Hilbert $A$-modules $N'=M'$. However, in general the C*-algebras of ''compact'' modular operators $K_A(N)$ and $K_A(M)$ are not a pair of a C*-algebra and one of its norm-closed two-sided ideals if considered as C*-subalgebras of $End_A(N')=End_A(M')$. Moreover, bounded (adjointable) module operators on $M$ extended to the monotone complete C*-algebra $End_A(N')=End_A(M')$ and, afterwards, reduced to bounded module operators with domain $N$ might not preserve $N$. Consequently, if for pairs $I \subseteq J$ of a two-sided norm-closed ideal $I$ in a C*-algebra $J$ we canonically have ${\mathbb M}(J) \subseteq {\mathbb M}(I)$ (cf.~local multiplier algebra definition for C*-algebras, \cite[Prop.~1.2.20, Def.~2.3.1, Prop.~2.3.4]{Ara_Mathieu}), we do not know anything definit on the interrelation of the multiplier C*-algebras $End_A^*(M)$ and $End_A^*(N)$ except that they share $K_A(M) \oplus Z({\mathbb M}(K_A(M)) {\rm id}_N$ by Lemma \ref{lemma_basic}. However, in case $N$ is selfdual we have an isometric $*$-representation of $End_A^*(M)$ in $End_A(N')$  as an order-dense C*-subalgebra with the same identity operator and center. }
\end{rmk}


\section{Special bounded modular functionals and kernels of bounded modular operators}

For bounded adjointable operators $T: M \to N$ between Hilbert $A$-modules $M$, $N$ over a C*-algebra $A$ there are some simple facts characterizing basic situations. Obviously, $T$ is $A$-linear, the kernel ${\rm ker}(T)$  and the set $T^*(N)$ are orthogonal to each other in $M$ and norm-closed.

\begin{lem}
Let $M$, $N$ be two Hilbert $A$-modules over a C*-algebra $A$, and let $T: M\to N$ be a bounded adjointable operator.  Then:
\begin{enumerate}
\item The kernel ${\rm ker}(T)$ of $T$ is biorthogonally complemented. 
\item There does not exist any non-zero element of $M$ orthogonal to both the sets ${\rm Ker}(T)$ and $T^*(N)^{\bot\bot}$. 
\item The direct orthogonal sum of ${\rm Ker}(T)$ and $(T^*(N))^{\bot\bot}$ might not be equal to $M$.
\end{enumerate}
\end{lem}

\begin{proof}
To show (i), suppose there exists an element $x \in {\rm ker}(T)^{\bot\bot}$ such that $T(x) \not= 0$. Then $x$ is orthogonal to $T^*(N)^\bot$, i.e.
$0=\langle x, T^*(y) \rangle = \langle T(x), y \rangle$ for any $y \in N$. Therefore, $T(x)=0$, a contradiction. 
To derive (ii), the argument is the same: Any element $x \in M$ orthogonal to $T^*(N)$ belongs to ${\rm ker}(T)^{\bot\bot}$, so there is no element of $M$ orthogonal to both ${\rm ker}(T)={\rm ker}(T)^{\bot\bot}$ and $(T^*(N))^{\bot\bot}$, as well as to their direct orthogonal sum in $M$. 

To show the possible non-coincidence of $M$ and of the direct orthogonal sum of ${\rm Ker}(T)$ and $T^*(N)^{\bot\bot}$ for a certain bounded adjointable operator between Hilbert C*-modules consider $A=C([0,1])$ as a Hilbert $A$-module over itself, and the element $f_0 \in A$ such that $f_0$ equals to zero on $[0,\frac{1}{2}]$ and $f_0$ is strongly positive on $(\frac{1}{2},1]$. Then the multiplication operator $T$ of $A$ by $f_0$ has the property ${\rm Ker}(T) \oplus (T^*(N))^{\bot\bot} \not=M$.
\end{proof}

\begin{prop} \label{prop_consequence1}
Let $N$ be a Hilbert $A$-module over a C*-algebra $A$. Suppose, there exists a bounded module operator $T_0: N \to N$ the kernel of which is not biorthogonally complemented. Then the biorthogonal complement of the kernel of $T_0$ with respect to $N$ admits a non-zero bounded $A$-linear functional $r_0: {\rm Ker}(T_0)^{\bot\bot} \to A$ such that $r_0$ is the zero functional on ${\rm Ker}(T_0)$. The map $r_0$ can be extended to a bounded $A$-linear functional on $N$. 
\end{prop}

\begin{proof}
In the situation given ${\rm Ker}(T_0)^{\bot\bot}$ contains an element $x_0 \not= 0$ that does not belong to ${\rm Ker}(T_0)$ and for which $T_0(x_0) \not= 0$. Consequently, the bounded $A$-linear functional $r_0(.) := \langle T_0(x_0), T_0(.) \rangle: N \to A$ maps ${\rm Ker}(T_0)$ to zero, but is unequal to zero on ${\rm Ker}(T_0)^{\bot\bot}$ since $\langle T_0(x_0), T_0(x_0) \rangle > 0$ by assumption. Obviously, the defined map $r_0$ can be applied to any element of $N$.
\end{proof}

\begin{prop} \label{prop_consequence2}
Let $M \subset N$ be a pair of Hilbert $A$-modules over a C*-algebra $A$ such that $M^\bot = \{ 0 \}$ and $N$ is full. Suppose, there exists a non-trivial bounded module functional $r_0: N \to A$ the kernel of which contains $M$. Then $N$ admits a bounded non-adjointable module operator $T_0: N \to N$ such that its kernel is not biorthogonally complemented and contains $M$.
\end{prop}

\begin{proof}
Note, that any C*-algebra $A$ can be considered as a Hilbert $A$-module over itself, so any bounded $A$-linear functional $r: M \to A$ can be considered as a bounded $A$-linear module map. Thus, any pair of non-self-dual Hilbert $A$-modules $M,N$ with $M \subset N$, $M^\bot = \{ 0 \}$ and an existing non-zero bounded $A$-linear functional $r_0: N \to A$ vanishing on $M$ gives rise to ${\rm Ker}(r_0) \not= {\rm Ker}(r_0)^{\bot\bot}$ inside $N$. Consider the set of operators $\{ T_{0,x,a} : T_{0,x,a}(\cdot) = x ( a \cdot r_0(\cdot)) \,\, {\rm with} \,\, x \in N, a \in A \}$. The set $\{ a \cdot r_0(z) : z \in N, a \in A \}$ forms a two-sided ideal in $A$ by the right $A$-linearity of $r_0$, by the right $A$-module property of $N$, by the free choice of $a \in A$ and by the supposition $\langle N,N \rangle = A$. In case $\langle x,x \rangle$ would be orthogonal to the set  $\{ a \cdot r_0(z) : z \in N, a \in A \}$ for any $x \in N$ the two-sided ideal $\{ a \cdot r_0(z) : z \in N, a \in A \}$ would have to consist only of the zero element of $A=\langle N,N \rangle$, a contradiction to the non-triviality of $r_0$. So, the set  $\{ T_{0,x,a} : T_{0,x,a}(\cdot) = x (a \cdot r_0(\cdot)) \,\, {\rm with} \,\, x \in N, a \in A\}$ contains at least one non-zero bounded non-adjointable module operator $T_0$ on $N$ with a not biorthogonally complemented kernel containing ${\rm Ker}(r_0) \supseteq M$.
\end{proof}

Therefore, for a given C*-algebra $A$ and a given full Hilbert $A$-module $N$ the problem of the existence of bounded $A$-linear functionals $r_0$ possessing a not biorthogonally closed kernel that admits a trivial orthogonal complement is tightly connected to the problem of the existence of bounded $A$-linear operators $T_0$ on $N$ possessing a not biorthogonally closed kernel that admits a trivial orthogonal complement. 

\begin{thm} (cf.~\cite[Lemma 2.4]{F_2002})
The kernel of any bounded $A$-linear operator between two Hilbert $A$-modules over a monotone complete C*-algebra $A$ is biorthogonally complemented. 
\end{thm}

The proof is a combination of Theorem \ref{theorem_ok} and Proposition \ref{prop_consequence1}. In total, we finally found a correct proof of \cite[Lemma 2.4]{F_2002} for monotone complete C*-algebras and for compact C*-algebras. The statement \cite[Lemma 2.4]{F_2002} is false in the general C*-case by \cite{KS}.


\section{Extension of $C^*$-linear functionals: the compact C*-case}

\bigskip
Among the C*-algebras $A$ for which the category of Hilbert C*-modules over them admits most of the positive properties one needs for easy applications, is the class of compact C*-algebras, i.e. those that admit a faithful $*$-representation in some C*-algebra of compact operators on a suitable Hilbert space. A similar result from a slightly different point of view has been obtained by V.~M.~Manuilov, cf.~\cite[Thm.~10]{M}.

\begin{thm}
Let $A$ be a compact C*-algebra and $M \subseteq N$ be two Hilbert $A$-modules. Suppose that $M \subseteq N$ has the orthogonal complement $M^\bot_N=\{0\}$ with respect to $N$.
Then the bounded $A$-linear zero functional of $M$ to $A$ admits a unique extension to a bounded $A$-linear functional of $N$ to $A$ such that it vanishes on $M$ -- the zero functional, because $M=N$. The more, the kernel of every bounded $A$-linear operator mapping it to another Hilbert $A$-module is biorthogonally complemented in it. 
\end{thm}

\begin{proof} 
For a compact C*-algebra $A$ and any Hilbert $A$-module over it we have two nice properties: $\,$(i) There exists an orthogonal basis of it, i.e. a generating set $\{ x_\alpha : \alpha \in I \}$ in it such that $\langle x_\alpha, x_\beta \rangle = 0$ for any $\alpha \not= \beta$ in $I$ and $\langle x_\alpha, x_\alpha \rangle = p_\alpha$ with $p_\alpha$ an non-zero atomic projection in $A$ (cf.~\cite[Thm.~2, Thm.~4]{BG_2002}); $\,$(ii) the Hilbert $A$-module is an orthogonal direct summand whenever it is isometrically embedded into another Hilbert $A$-module as a Hilbert $A$-submodule (cf.~\cite{Magajna, Guljas_2019}). Therefore, $M$ is a direct orthogonal summand of $N$ with trivial orthogonal complement, i.e. $M$ and $N$ coincide. This forces the isometric coincidence of their $A$-dual Banach $A$-modules. 

The kernel of any bounded $A$-linear operator on the Hilbert $A$-module under consideration is a Hilbert $A$-submodule of it, and hence, an orthogonal direct summand of it (cf.~\cite{Magajna}). In particular, it coincides with its biorthogonal complement. 
\end{proof}


\section{The situations for one-sided maximal modular ideals of general C*-algebras and for other norm-closed ideals}

\bigskip
Let us return to the class of examples of right norm-closed maximal ideals $D$ of C*-algebras $A$. We would like to consider the class of modular right maximal ideals $D \subset A$, i.e. of such ideals for which there exists an element $u_D \in A$ such that $(a-u_Da) \in D$ for any $a \in A$. They are automatically norm-closed, cf. \cite[Thm.~1.3.1]{Murphy}. For C*-algebras even more is known: every maximal right ideal of a C*-algebra is norm-closed (\cite[Cor.~3.6]{CGDRP}), every maximal right ideal of a Fr\'echet algebra is modular (\cite[Thm.~2.2.42, Prop.~4.10.23]{D}), and in the commutative C*-case the codimension is 1 (\cite[Thm.~2.4(i)]{CGDRP}, Gel'fand-Mazur theorem).

\begin{thm}
Let $A$ be a C*-algebra and $D$ be a right modular maximal ideal in $A$. Consider them as Hilbert $A$-modules in the usual way transferring the algebraic properties of $A$ to modular and $A$-valued inner product structures on both $A$ and $D \subset A$. Then the zero functional on the Hilbert $A$-submodule $D$ has only the zero modular functional on its biorthogonal complement $D^{\bot\bot}$ in the hosting Hilbert $A$-module $A$ as its continuation.
\end{thm}

\begin{rmk}  {\rm
In case the C*-algebra $A$ contains a finite-dimensional (matrix) block and the left annihilator projection $p \in A^{**}$ of $D$ is an atomic projection in this finite-dimensional block of $A \subseteq A^{**}$, then $D^{\bot\bot}=(1-p)A \subset A$ (but of $A^{**})$. Otherwise, the left annihilator projection of $D$ in $A^{**}$ is not an element of $A$, and $D^{\bot\bot}=A$. }
\end{rmk}

\begin{proof}
By \cite[Thm.~5.3.5]{Murphy} there exists a bijection between the set of pure states $\rho$ on $A$ and the set of all modular maximal right ideals $N_\rho$ of $A$, where $N_\rho = \{ a \in A : \rho(aa^*)=0 \}$. The inverse mapping is induced by a factorization of $A$ by a modular maximal right ideal $D$ resulting in a one-dimensional $*$-representation of $A$ with a cyclic vector that induces the pure state.  Considering the universal $*$-representation $\pi_u$ of $A$ on a Hilbert space $H_u$ as defined in \cite[3.7.6]{Pedersen}, these one-dimensional $*$-representations of $A$ are direct summands of $\pi_u$. So the bicommutant of $\pi_u(A)$, a von Neumann algebra, containes atomic projections $p_\rho$ in its type I part that project onto these related one-dimensional $*$-representation spaces and realize any pure state $\rho$ on $A$ that way. The more, the enveloping von Neumann algebra $\pi_u(A)''$ of $A$ is isomorphic, as a Banach space, to the second dual Banach space $A^{**}$ of $A$, cf. \cite[Thm.~3.7.8]{Pedersen}. That is, we have a one-to-one relation between the set of all modular maximal right ideals $N_\rho$ of $A$ and atomic projections $p_\rho$ of $A^{**}=\pi_u(A)''$ as $\pi_u(N_\rho)^{\bot\bot} = (1-p_\rho)A^{**}$ in the sense of taking the left annihilator of $\pi_u(N_\rho)$ and then the right annihilator. 

So the initial problem of the assertion above translates into the question of the nature of the intersection of $\pi_u(A) p_\rho$ with $A^{**} p_\rho$. Since $p_\rho A^{**} p_\rho= {\mathbb C}$, either $p_\rho \in \pi_u(A)$, i.e. in the (existing in this case) matrix part of $A$, and $p_\rho \pi_u(D)$ as a Hilbert $A$-submodule of $p_\rho \pi_u(A)$ is a direct orthogonal summand of $p_\rho \pi_u(A)$, or $\pi_u(A) p_\rho \cap A^{**} p_\rho = \{ 0 \}$. In the first case $p_\rho \pi_u(D)$ is self-dual, so it contains already all modular continuations of the zero functional to its biorthogonal complement in $p_\rho\pi_u(A)$. In the second case we end up with the zero element as the unique continuation. So the result follows.
\end{proof}

\begin{prop} \label{prop_commutative}
Let $A$ be a commutative C*-algebra and $I \subseteq J$ be two essential norm-closed ideals of $A$. Then:
\begin{enumerate}
\item The left and the right annihilator sets of $I$ w.r.t. $J$ and of $I$ or $J$ w.r.t. $A$ consist only of the zero element. 
\item The multiplier algebra of $J$ is isometrically $*$-algebraically represented in the multiplier algebra of $I$, i.e. ${\mathbb M}(J) \subseteq {\mathbb M}(I)$. 
\item Suppose, $I$ and $J$ are considered as (right) Hilbert $A$-modules. Then any bounded $A$-linear functional from $J$  to $A$ which is supposed to be the zero functional on $I$ equals to zero on $J$.
\end{enumerate}
\end{prop}

\begin{proof}
Start with the basic situation of two essential norm-closed ideals $I \subseteq J$ of $A$. 
Since norm-closed ideals of a commutative C*-algebra are C*-algebras themselves, any one-sided annihilator of $I$ or $J$ forces its adjoint to be an one-sided annihilator of $I$ or $J$ from the other side, simply by commutativity of $A$. The property of $I$ and of $J$ to be essential norm-closed ideals of $A$ yields one-sided annihilators of them w.r.t. $J$ and/or to $A$ to be the set $\{ 0_A \}$. By commutativity of $A$ any one-sided multiplier of $I$ or $J$ is in fact a two-sided multiplier.

Moreover, one has ${\mathbb M}(J) \subseteq {\mathbb M}(I)$ in the sense of an injective $*$-algebraical inclusion of C*-algebras, cf.~\cite[Prop.~1.2.20, Def.~2.3.1, Prop.~2.3.4]{Ara_Mathieu}. Any bounded $A$-linear functional $r: J \to A$ can be described as the multiplication by a fixed element $n_r \in {\mathbb M}(J)$ since the $A$-dual Banach $A$-module of $J$ can be represented as ${\mathbb M}(J)$ in that way. Therefore, $n_r$ can be seen as a multiplier of $I$, too. So, any bounded $A$-linear functional $r$ on $J$ vanishing on $I$ leads to an annihilator of $I$ w.r.t. $J$. And this can be only the zero element of $A$ by supposition. 
\end{proof}

We cannot transfer the proof to the non-commutative situation straightforwardly, since the $A$-dual Banach $A$-module $I'$ of $I$ has to be identified with the right multiplier algebra ${\mathbb R}{\mathbb M}(I)$ of $I$ which might be larger than ${\mathbb M}(I)$ and, therefore, might be not invariant w.r.t. involution, i.e.~might be not a $*$-algebra.


\smallskip
The search for counterexamples of special C*-functionals in the class of monotone complete C*-algebras $A$, setting $N=A$, and right norm-closed ideals $D$ in $A$, setting $M=D$, is not successful. Basically, the left annihilator of $D$ is always of the form $Ap$ with $p \in A$, a fact from AW*-algebra theory (cf.~\cite[Prop.~8.2.2]{SaitoWright_2015}). And additionally, $D$ is order-dense in $(1-p)A$. Moreover, $(1-p)A$ and $A$ are right self-dual Hilbert $A$-modules. So we can formulate a corollary.

\begin{cor}
Let $A$ be a monotone complete C*-algebra and $D$ be a right, norm-closed ideal of $A$ which is order dense in $A$. Then there does not exist any non-zero bounded $A$-linear functional $r_0: A \to A$ such that $r_0(D)= \{ 0 \}$. 
\end{cor}

Set $N=A$ and $M=D$, recall $D^{\bot} =\{ 0 \}$ and apply Theorem \ref{theorem_ok}. We get the desired result. Going further we consider a pair of two ideals in a monotone complete C*-algebra.

\begin{cor}
Let $A$ be a monotone complete C*-algebra, let $I$ and $J$ be two right norm-closed ideals of $A$ such that $I \subset J \subseteq A$, both with the orthogonal complement $pA$, $p=p^2 \geq 0$ in $A$. Switching to a consideration of $I$ and $J$ as Hilbert $A$-submodules of $A$, the orthogonal complement of $I$ with respect to $J$ is equal to $\{ 0 \}$. Then the right ideals $I \subset J$ have the same left carrier projection $(1-p) \in A$, and as Hilbert $A$-submodules of $A$ they fulfil $I' = J' = (1-p)A$. In particular, $I$ and $J$ are order dense in $(1-p)A$, or equivalently in the picture of Hilbert $A$-modules, $\tau^o_2$-dense. So, there does not exist any bounded $A$-linear functional $r_0: J \to A$ such that $r_0$ vanishes on $I$, but which is not the zero functional on $J$.
\end{cor}

\begin{proof}
Consider the two ideals as (right) Hilbert $A$-modules $M=I$ and $N=J$ in $A$. The right annihilator of $I$ and of $J$ w.r.t. $A$ equals to zero since they are right ideals of $A$ and right $A$-modules. Suppose, the left annihilator of $I$ w.r.t. $A$ is $pA$ and of $J$ w.r.t. $A$ is $qA$. Since $I \subset J$ we have $p \geq q$. Switch to the picture of a Hilbert $A$-modules $M\subset N$. In case $p > q$ the subset $(p-q)N$ is non-zero and orthogonal to $M$, a contradiction to the supposition. So $p=q$. By Theorem \ref{theorem_ok} $M'=N'=(1-p)A$ as Hilbert $A$-modules, and the uniqueness of the continuation of the zero functional from $I$ to $J$ follows. 
\end{proof}

Summing up, we found results for the zero functional continuation problem for pairs of a Hilbert C*-submodule with trivial orthogonal complement in another Hilbert C*-module, with quite different roots. So the research about this problem has to be continued. 


\section*{Acknowledgement}

I am indebted to Vladimir M.~Manuilov and Evgenij V.~Tro\H{\i}tsky for long-lasting valuable discussions on C*-algebraic, differential-geometric and topological problems closely related to those treated in the present article, as well as for carefully examining earlier attempts to settle the zero functional continuation problem in cases. I am grateful to Orr Moshe Shalit and Michael Skeide whose persistent questions forced me to find classes of Hilbert C*-modules for which the discussed above problem can be waived. My graditute goes to Boris Gulja{\v{s}} and to Michael Skeide for reading the manuscript at an earlier stage of evolvement. 
The shown possible appearance of this kind of problems is still a source of uncertainty to me about earlier results in Hilbert C*-module theory what requires repeated reading of many prior publications. At the same time, the new points of view stimulate further research, for example, on non-adjointable module operators aka left multiplier operators.


\end{document}